\documentclass[12pt]{amsart}%,reqno
\usepackage{amssymb,amsmath,amsthm}
\usepackage{a4}
\usepackage{color}
\usepackage[colorlinks=true,
linkcolor=webgreen,
filecolor=webbrown,
citecolor=webgreen]{hyperref}
\definecolor{webgreen}{rgb}{0,0,1}%{1,0.0,.6}
\definecolor{recrown}{rgb}{1,.2,.6}
\usepackage{fullpage}
\begin{document}
\newtheorem{theorem}{Theorem}
\newtheorem{corollary}[theorem]{Corollary}
\newtheorem{lemma}[theorem]{Lemma}
\theoremstyle{definition}
\newtheorem*{example}{Examples}
\newtheorem{conjecture}[theorem]{Conjecture}
\newtheorem{thmx}{\bf Theorem}
\renewcommand{\thethmx}{\text{\Alph{thmx}}}% "letter-numbered" theorems
\newtheorem{lemmax}{Lemma}
\renewcommand{\thelemmax}{\text{\Alph{lemmax}}}% "
%\leftmargin=.5in
%\rightmargin=0.5in
%\textwidth=6truein
%\textheight=11.6truein
\hoffset=-0cm
%\voffset=+-2cm
\theoremstyle{definition}
\newtheorem*{definition}{Definition}
\theoremstyle{remark}
\newtheorem*{remark}{\bf Remark}
\title{\bf On a factorization result of \c{S}tef\u{a}nescu}
\author{Sanjeev Kumar$^{\dagger}$}
\address{$~^\ddagger$ Department of Mathematics, SGGS College, Sector-26, Chandigarh-160019, India}
\author{Jitender Singh$^{\ddagger}$}
\address{$~^\dagger$ Department of Mathematics, Guru Nanak Dev University, Amritsar-143005, India}
%\markright{}
\date{}
\maketitle
\parindent=0cm
\footnotetext[3]{sanjeev\_kumar\_19@yahoo.co.in}
\footnotetext[2]{$^{,*}$Corresponding author: jitender.math@gndu.ac.in; sonumaths@gmail.com

}
\begin{abstract}
In this article, some factorization properties of polynomials over discrete valuation domains are elucidated. These properties along with the notion of Newton index of a polynomial leads to a generalization of the main result proved by \c{S}tef\u{a}nescu [`On the factorization of polynomials over discrete valuation domains', \emph{Versita} \textbf{22}:1  (2014), 273--280].
\end{abstract}
The augean work of identifying irreducible polynomials in a prescribed domain is a classical yet exciting theme as is evident from the fact that the prolific irreducibility criteria due to Sch\"onemann \cite{Sc} and Eisenstein \cite{E} have witnessed exquisite extensions and generalizations for  decades. Recently, Weintraub \cite{W} generalized Eisenstein's criterion and also provided a correction to the false claim made by Eisenstein himself \cite{E}. However, one of the earliest known generalization is credited to Dumas \cite{Du}, who generalized the Sch\"onemann-Eisenstein  criteria by establishing an intimate connection between the valuation theoretic properties of the polynomials having integer coefficient and the associated geometric properties exhibited by the Newton polygons.
%who connected the irreducibility of a polynomial with the geometry of the associated Newton polygon or Newton index. %This approach is interesting in its own right, and it easily translates to the powerful algebraic machinery of valuations.

Let $(R,v)$ be a discrete valuation domain  and $f=a_0+a_1x+\cdots+a_nx^n \in {R}[x]$ be a nonconstant polynomial. For any index $i\in \{0,1,\ldots,n-1\}$, let $m_i(f)$ be the slope of the line segment joining the points $(n,v(a_n))$ and $(i,v(a_i))$, that is,
\begin{eqnarray*}
m_i(f)&=& \frac{v(a_n)-v(a_i)}{n-i}.
\end{eqnarray*}
The Newton index $e(f)$ of $f$ is then defined to be the number $\max_{1\leq i\leq n-1}\{m_i(f)\}$ and the following identity
\begin{eqnarray*}
e(f_1f_2)=\max\{e(f_1),e(f_2)\} ~\text{for all}~ f_1,~f_2\in {R}[x],
\end{eqnarray*}
is well acknowledged. In \cite[Theorem 2.2]{S}, some factorization properties for polynomials over a discrete valuation ring $(R,v)$ were investigated  and the following main result was established.%, which we state for the case of $\mathbb{Z}[x]$ and the valuation $v_p$. %\cite[Theorem 2.2., p.~275]{S}.
\begin{thmx}\label{thA}
Let $(R,v)$ be a discrete valuation domain and $f=a_0+a_1x+\cdots+a_nx^n \in R[x]$. If there exists an index $s\in \{0,\ldots,n-1\}$ such that
\begin{enumerate}
    \item[(a)] $v(a_n)=0$; $m_i(f)<m_s(f)$ for all $i=0,1,\ldots,n-1$, $i\neq s$,
    \item[(b)] %If $s=0$, then $\gcd(n,v_p(a_s))=1$, otherwise,
    $n(n-s)(m_0(f)-m_s(f))=-1$,
\end{enumerate}
then the polynomial $f$ is either irreducible in $R[x]$, or has a factor whose degree is a multiple of $n-s$.
\end{thmx}
The hypothesis (a) in Theorem \ref{thA} along with the additional requirement of coprimality of $v(a_s)$ and $n-s$ yields the factorization result of Weintraub \cite{W}, that is, if $f(x)=g(x)h(x)$ is any factorization of the polynomial $f$ in $R[x]$, then $\min\{\deg g, \deg h\}\leq s$ or equivalently
\begin{eqnarray*}
\max\{\deg g, \deg h\}\geq (n-s).
\end{eqnarray*}
The hypothesis (b) in Theorem \ref{thA} enables us to proceed one step ahead of the aforementioned inequality by asserting $n-s$ as a divisor of $\deg g$ or $\deg h$.

However, incase if $R=\mathbb{Z}$ and $v=v_p$ the associated $p$-adic discrete  valuation of $\mathbb{Z}$ with $s=0$, then the assumption (b) in the hypothesis of Theorem \ref{thA} becomes void. So, in this particular context, an amelioration of the above result is necessary. In this article, we modify the hypothesis (b) of Theorem \ref{thA} in order to allow $s=0$ without ambiguity and then prove a mild generalization, which we hope paves the way for a significant criterion that rectifies an error Eisenstein made himself, that too with a weaker hypothesis. More precisely, we have the following result.
\begin{theorem}\label{th1}
Let $(R,v)$ be a discrete valuation domain and $f=a_0+a_1x+\ldots+a_nx^n \in R[x]$ with $v(a_n)=0$. Let there be an index $s\in \{0,\ldots,n-1\}$ such that
\begin{enumerate}
    \item[(a)]\label{a} $m_i(f)<m_s(f)$ for all $i=0,1,\ldots,n-1$, $i\neq s$,
    \item[(b)]\label{b}   $d_s=\gcd(n-s,v(a_s))$ satisfies
    \begin{eqnarray*}
  (-d_s) =
    \begin{cases}
    n(n-s)(m_0(f)-m_s(f)),&~\text{if}~ s\neq 0;\\
     -1,&~\text{if}~ s=0.
    \end{cases}
    \end{eqnarray*}
\end{enumerate}
Then any factorization $f(x)=f_1(x)f_2(x)$ of $f$ in ${R}[x]$ has a factor whose degree is a multiple of $(n-s)/d_s$. In particular, if $s=0$, then $f$ is irreducible.
\end{theorem}
%The irreducibility of the following polynomials
%\begin{example}
%(1). Consider the polynomial
%\begin{eqnarray*}
%X=243+27 x+81x^2+9x^3-128x^4,
%\end{eqnarray*}
%where $n=4$; $a_0=243$, $a_1=27$, $a_2=81$, $a_3=27$, $a_4=-128$. Taking $s=1$ and $p=3$, we have $v_3(a_4)=0$. Also, $v_3(a_s)=v_3(27)=3$ and $v_3(a_0)=v_3(3^5)=5$. Here, $i=0,2,3$, so that
% $\gamma_1(X)(0)=3(4-v_3(243))=-3$, $\gamma_1(X)(2)=3(2-v_3(81))=-6$, $\gamma_1(X)(3)=3(1-v_3(27))=-6$ all of which are negative, where $-\gamma_1(X)(0)=3=\gcd(3,3)=\gcd(4-1,v_3(a_1))$.
%Thus, the polynomial $X$ satisfies the hypotheses of Theorem \ref{th1}; and so, $X$ is irreducible, or $X$ has a factor of degree $\geq 1$ in $\mathbb{Z}[x]$.
%\end{example}
The examples below justify that there exist polynomials whose factorization properties can be deduced using Theorem \ref{th1} wherein Theorem \ref{thA} fails to be applicable.
\begin{example} Let $p$ be a prime number and $v=v_p$, the discrete valuation on $\mathbb{Z}$.

\textbf{(1).}  Consider the polynomial
\begin{eqnarray*}
X(x)=(a_0+p^2(p-1)a_2x^2)p^{n-2}+p^{n-3}a_1x+a_nx^{n}\in \mathbb{Z}[x],~n\geq 5,
\end{eqnarray*}
where $n$ is odd, $a_i>0$  and $v_p(a_i)=0$ for $i\in\{0,1,2,n\}$. Taking $s=1$, we observe that
\begin{eqnarray*}
m_i(X)-m_1(X)\leq -\frac{n-2}{n}+\frac{n-3}{n-1}=\frac{-2}{n(n-1)}<0,
\end{eqnarray*}
and also $n(n-1)(m_0(X)-m_1(X))=-2=-d_1$. By Theorem \ref{th1} either $X$ is irreducible or $X$ has a factor whose degree is a multiple of $(n-1)/2$.

\textbf{(2).} The polynomial
\begin{eqnarray*}
Y(x)=(1+p^2 x^3)p^{n-2}+p^{n-4}x^2+x^{n}\in \mathbb{Z}[x],~n\geq6,
\end{eqnarray*}
where $n$ is even  satisfies the hypothesis of Theorem \ref{th1} with $s=2$, $a_0=p^{n-2}$, $a_2=p^{n-4}$, $a_3=p^{n+1}$; $a_i=0$ for $i\not\in\{0,2,3,n\}$;
$m_i(Y)-m_2(Y)\leq \frac{-2}{n(n-2)}<0$, and $n(n-1)(m_0(Y)-m_1(Y))=-2=-d_1$.
So, either $Y$ is irreducible or $Y$ has a factor of degree equal to a multiple of $(n-2)/2$.

\textbf{(3).} Let $n>2$ and $d$ be a divisor of $n-1$. Now consider the polynomial
\begin{eqnarray*}
Z(x,y)=a_0(x)+a_1(x)y+y^{n}\in \mathbb{Z}[x,y],
\end{eqnarray*}
where each of $a_0(x)$ and $a_1(x)$ is an irreducible polynomial of degree $d$ in $\mathbb{Z}[x]$. For nonzero  $g\in \mathbb{Z}[x]$, define  $v(g)=-\deg(g)$ and $v(0)=\infty$.  Taking $s=1$, we have
\begin{eqnarray*}
m_i(Z)-m_1(Z)\leq \frac{d}{n}-\frac{d}{n-1}=\frac{-d}{n(n-1)}<0,
\end{eqnarray*}
and also $n(n-1)(m_0(Z)-m_1(Z))=-d=-\gcd(v(a_0(x)),n-1)$. By Theorem \ref{th1}, $Z$ has a factor (with respect to $y$) whose  degree is a multiple of $(n-1)/d$. This example can also be viewed in $K[x,y]$ where $K$ is any field of characteristic zero, and $v$ is the discrete degree valuation on $K$.
%(4). For a prime $p$, the polynomial
%\begin{eqnarray*}
%T(x)=p^{n+1}(1+x+\cdots+x^{s-1})+p^{n-s}x^s+x^{n}\in \mathbb{Z}[x],~n-s>s\geq 1.
%\end{eqnarray*}
%satisfies  $m_i(T)-m_1(T)\leq -\frac{n+1}{n}+\frac{n-s}{n-s}=\frac{-1}{n}<0$, and $n(n-s)(m_0(X)-m_s(X))=-(n-s)=-d_s$. So, by Theorem \ref{th1} $U$ has a factor of degree equal to a multiple of $(n-s)/(n-s)=1$.
%(4). Let $K$ be a field of characteristic zero. Let $n\geq 4$ and $d$ be a divisor of $n-1$. Then the polynomial
%\begin{eqnarray*}
%T(x,y)=(1+x+x^{d})(1+y)+(1+x^{d-2})y^{2}+y^n\in \mathbb{K}[x,y],~n\geq 4.
%\end{eqnarray*}
%satisfies the hypothesis of Theorem \ref{th1} with $a_0(x)=1+x^2+x^{d}=a_1(x)$, $a_2(x)=1+x^{d-2}$, $a_n(x)=1$ for $s=1$ using the discrete degree valuation of $\mathbb{K}[x]$. So, $Z$ has a factor (with respect to $y$) whose degree is a multiple of $(n-1)/d$.
%(3). The polynomial
%\begin{eqnarray*}
%Z(x)=p(1+a_1 p x+\cdots a_{n-1}(px)^{n-1})+x^n, \in \mathbb{Z}[x],~0<a_i<p,~(i=1,\ldots,n-1),
%\end{eqnarray*}
%for $n>2$ satisfies the hypothesis of Theorem \ref{th1} with $s=0$ and $d_0=1$. Here $Z$ is always irreducible, since $Z$ has a factor of degree as a multiple of $n$.
\end{example}
\begin{proof}[\bf Proof of Theorem \ref{th1}.] For $s=0$, Theorem \ref{th1} reduces to the classical Eisenstein's irreducibility criterion. So, let $s>0$. Let $k_i=\deg f_i$, $x_i=v(f_i(0))$ for $i=1,2$. Let $y_j=v(a_j)$ for each $j=0,1,\ldots,n-1$. With these notations, the hypothesis (b) in the theorem becomes
\begin{equation}\label{1}
 (n-s)y_0-ny_s=d_s.
\end{equation}
By  hypothesis (a) in the theorem, we have $e(f)=-v(a_s)/(n-s)=-y_s/(n-s)$ and since $e(f)\geq e(f_i)$, we deduce that
\begin{eqnarray*}
-y_s/(n-s)=e(f) \geq e(f_i)\geq  {-v(f_i(0))}/{k_i}=-x_i/k_i.
\end{eqnarray*}
So we must have $y_sk_i\leq (n-s)x_i$ for $i=1,2$. We also have
\begin{eqnarray*}
y_0=v(f(0))=v(f_1(0))+v(f_2(0))=x_1+x_2;~n=k_1+k_2.
\end{eqnarray*}
Using  these in \eqref{1}, we get that
\begin{equation*}\label{2}
0\leq (n-s)x_2-k_2y_s=(n-s)x_2+k_1y_s-n y_s\leq   (n-s)(x_2+x_1)-n y_s=d_s.
\end{equation*}
Consequently, if we let  $\kappa=((n-s)/d_s)x_2-k_2(y_s/d_s)$, then $0\leq \kappa \leq 1$, and so,  either  $\kappa=0$ or $\kappa=1$.

 If $\kappa=0$, then in view of (b), the number ${(n-s)}/{d_s}$ must divide $k_2$, and so, $k_2$ must be a multiple of ${(n-s)}/{d_s}$.
On the other hand, if $\kappa=1$, then
\begin{eqnarray*}
((n-s)/d_s)(y_0-x_1)-(n-k_1)(y_s/d_s)=1,
\end{eqnarray*}
which in view of \eqref{1} gives $((n-s)/d_s)x_1=k_1(y_s/d_s)$, and so, the number ${(n-s)}/{d_s}$ divides $k_1$. In this case, $k_1$ is a multiple of ${(n-s)}/{d_s}$. Thus, in either of the cases, the degree of one of $f_1$ or $f_2$ is a multiple of ${(n-s)}/{d_s}$. This completes the proof.
\end{proof}
Most of the factorization results on polynomials over a discrete valuation domain $(R,v)$ require to take $v(a_s)$ and $n-s$ to be coprime, whenever $s$ is the smallest index for which the minimum  of the quantity $v(a_i)/(n-i)$, $0\leq i\leq n-1$ is $v(a_s)/(n-s)$ (See for example, Jhorar and Khanduja \cite{Jh}).  In the case when $v(a_s)$ and $n-s$ are not coprime, we have the following result.
\begin{lemma}\label{L2}
Let $(R,v)$ be a discrete valuation domain and $f=a_0+a_1x+\ldots+a_nx^n \in R[x]$ with $v(a_n)=0$. If $d_s=\gcd(v(a_s),n-s)>1$, then any factorization $f(x)=g(x)h(x)$ of $f$ in $R[x]$ has  $\max\{\deg g,\deg h\}\geq (n-s)/d_s$.
\end{lemma}
\begin{proof}
Assume on the contrary that $\max\{\deg g,\deg h\}<(n-s)/d_s$. Then $\deg g<(n-s)/d_s$ and $\deg h<(n-s)/d_s$. Since $f(x)=g(x)h(x)$, we have
\begin{eqnarray*}
n=\deg f=\deg g+\deg h<2(n-s)/d_s\leq 2n/d_s,
\end{eqnarray*}
which yields $d_s<2$, and so, $d_s=1$. This contradicts the hypothesis.
\end{proof}
\begin{theorem}\label{th2}
Let $(R,v)$ be a discrete valuation domain and $f=a_0+a_1x+\ldots+a_nx^n \in R[x]$ with $v(a_n)=0$. Let there be an index $s\in \{0,\ldots,n-1\}$ such that $m_i(f)<m_s(f)$ for all $i=0,1,\ldots,n-1$, $i\neq s$ and $d_s=\gcd(n-s,v(a_s))$. Then any factorization $f(x)=g(x)h(x)$ of $f$ in ${R}[x]$ has $\max\{\deg{g},\deg h\}\geq (n-s)/d_s$.
\end{theorem}
\begin{proof} The theorem in the case $d_s=1$ follows from a result proved in Jhorar and Khanduja \cite[Theorem 1.2]{Jh} for more general domains. On the other hand if $d_s>1$, then  the theorem follows from Lemma \ref{L2}.
\end{proof}
%\subsection*{Acknowledgements} %The authors are thankful to the anonymous referees for careful reading and useful comments leading to an improvement of the article.

\end{document}